\RequirePackage{fix-cm}
\documentclass[smallextended]{svjour3}       
\smartqed  
\usepackage{graphicx}
\usepackage{amssymb}
 \usepackage{amsmath}

\newcommand{\R}{{\mathbb{R}}}

\begin{document}

\title{Canard Phenomenon in a  modified Slow-Fast  Leslie-Gower and Holling type scheme model
}


\author{B. Ambrosio       \and
        M.A. Aziz-Alaoui  \and R. Yafia 
}


\institute{B. Ambrosio, M.A. Aziz-Alaoui\at
              Normandie Univ, UNIHAVRE, LMAH, FR-CNRS-3335, ISCN, 76600 Le Havre, France
              \email{benjamin.ambrosio@univ-lehavre.fr}           
            \and 
            R. Yafia \at    Ibn Zohr University, Agadir, Le Havre        
}

\date{Received: date / Accepted: date}

\maketitle

\begin{abstract}
Geometrical Singular Perturbation Theory has been  successful  to investigate a broad range of biological problems with different time scales. The aim of this paper is to apply this theory to a predator-prey model of modified Leslie-Gower type for which we consider that  prey reproduces
mush faster than  predators. This naturally leads to introduce a small parameter $\epsilon$ which gives rise to a slow-fast system. This system has a special folded singularity which has not been analyzed in the classical work \cite{KS01}.  We use the blow-up technique to visualize the behavior near this fold point $P$. Outside of this region the dynamics are given by classical singular perturbation theory. This allows to quantify geometrically the attractive  limit-cycle with an error of $O(\epsilon$) and shows that it exhibits the \textit{canard} phenomenon while crossing $P$.
\end{abstract}

\section{Introduction}
\label{sec:intro}
In \cite{az03}, the authors introduced the following model:
\begin{equation}\label{DLsystemaz03}
    \left\{ \begin{array}{rcl}
    \dot{x}=\left(r_{1}-b_{1}x-\frac{a_{1}y}{x+k_{1}}\right)x,\\
    \\
    \dot{y}=\left(r_{2}-\frac{a_{2}y}{x+k_{2}}\right)y \hspace{1cm}
\end{array} \right.
\end{equation}
where $x$ represent the prey and $y$ the predator. This two species food chain model describes a prey population $x$ which serves as food for a predator $y$.
The model parameters $r_1,r_2,a_1, a_2, b, k_1$ and $k_2$ are assumed to be positive. They are defined
as follows: $r_1$ (resp. $r_2$) is the growth rate of prey $x$ (resp. predator $y$), $b_1$ measures the strength of competition among individuals of species $x$, $a_1$  (resp. $r_2$)
is the maximum value of the per capita reduction rate of $x$  (resp. $y$) due to $y$, $k_1$ (respectively, $k_2$) measures the extent to which environment provides protection to prey $x$ (respectively, to the predator $y$). There is a wide variety of natural systems which may be modelled by system \eqref{DLsystemaz03}, see \cite{Ha91,Up97}. It may, for example, be considered as a representation of an insect pest–spider food chain. 
Let us mention that the first equation of system \eqref{DLsystemaz03} is standard. The second equation is rather absolutely not standard. Recall that the Leslie-Gower formulation is based on the assumption that reduction in a predator population has a reciprocal relationship with per capita availability of its preferred food. This leads to replace the classical growing term ($+xy$) in Lotka-Volterra predator equation by a decreasing term ($-y^2$).  Indeed, Leslie introduced a predator prey model where the carrying capacity of the predator environment is proportional to the number of prey. These considerations lead to the following equation for predator $\dot{y}=r_2y(1-\frac{y}{\alpha x}).$ The term $\frac{y}{\alpha x}$ of this equation is called the Leslie–Gower term.  In case of severe scarcity, adding a positive constant to the denominator, introduces a maximum decrease rate, which stands for environment protection. Classical references include \cite{Le48,Le60,Ma73,RM63}.
  In order to simplify \eqref{DLsystemaz03}, we proceed to the following change of variables:\\
$u(r_{1}t)=\frac{b_{1}}{r_{1}}x(t)$,
$v(r_{1}t)=\frac{a_{2}b_{1}}{r_{1}r_{2}}y(t)$,
 $a=\frac{a_{1}r_{2}}{a_{2}r_{1}}$,
$\epsilon=\frac{r_{2}}{r_{1}}$,
$e_{1}=\frac{b_{1}k_{1}}{r_{1}}$,
$e_{2}=\frac{b_{1}k_{2}}{r_{1}}$, $t'=r_{1}t$.\\
For convenience, we drop the primes on $t$.
We obtain the following system:
\begin{equation}\label{DLsystem}
 \left\{ \begin{array}{rcl}
u_t&=&u\left(1-u\right)-\frac{auv}{u+e_{1}},\\
v_t&=&\epsilon
v\left(1-\frac{v}{u+e_{2}}\right).\hspace{2cm}
\end{array} \right.
\end{equation}
We assume here that the prey reproduces much faster than the predator, $i.e.$ $r_1>>r_2$, which implies that $\epsilon$ is small.
Note that there are special solutions: $u=0,v_t=\epsilon v(1-\frac{v}{e_2})$ and $v=0,u_t=u(1-u)$. Hence, the quadrant $(0\leq u \leq 1, v\geq 0)$ is positively invariant for \eqref{DLsystem}. We restrict our analysis to this quadrant.  
We also assume the following conditions which ensure the existence of a unique attractive limit-cycle for \eqref{DLsystem}:
\[ae_2<e_1, ae_2 \mbox{ not to close of } e_1,\]
and, 
\[u^*<\frac{1-e_1}{2}, u^* \mbox{ not to close of } \frac{1-e_1}{2},\]
where $u^*$ is solution of
\[u+e_2=\frac{1}{a}(1-u)(u+e_1).\]
 Under these asumptions there are 4 fixed points in the positive quadrant:
  \[P_1=(0,0), P_2=(0,e_2), P_3=(1,0), P_4=(u^*,g(u^*)),\]
  where
    \[g(u)=\frac{1}{a}(1-u)(u+e_1).\]
 \begin{figure}[ht]
 \centering
  \includegraphics[scale=0.7]{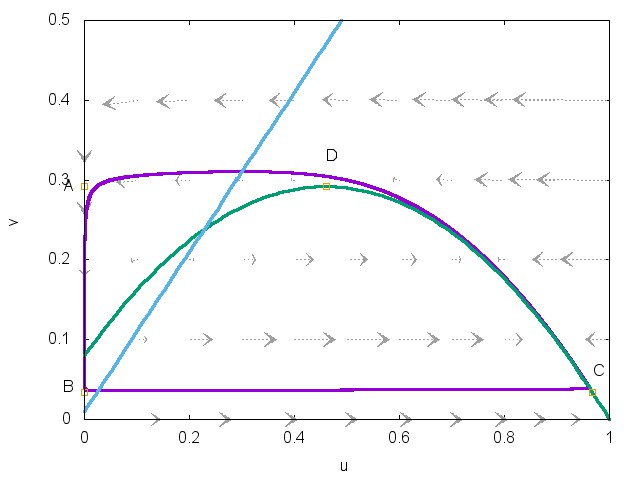}
  \label{fig:limcyclpdf}
\caption{Limit cycle and nullclines of system \eqref{DLsystem} for  $a=1$
, $e_{1}=0.08,$ and $e_{2}=0.01$, $\epsilon=0.01$}
\end{figure}
They also prevent additional singularities for the folded points.  Figure \ref{fig:limcyclpdf} illustrates nullclines and the attractive limit-cycle for \eqref{DLsystem}. 
Our aim is now to characterize the limit-cycle. In the following section we proceed to the classical slow-fast analysis which allows to describe the trajectories outside of a neighborhood of a special fold-point, induced by the nullcline $u=0$, which we will call $P$. In the third section, we use the blow-up technique to analyze the trajectories near this special fold point $P$.
 Now, let us fix a small value $\alpha>0$ and define a cross section $V=\{(u,v)\in \R^2; u>0, \, v=\frac{e_1}{a}+\alpha\}$. Then, by the regularity of the flow with regard to $\epsilon$,  the limit cycle crosses $V$ at a point $(k(\alpha)\epsilon+o(\epsilon),\frac{e_1}{a}+\alpha)$ (below, for convenience, we do not write the dependence on $\alpha$).  
We have the following theorem. Let 
\[\bar{u}=\frac{1-e_1}{2},\]
and 
\[A=(0,g(\bar{u})), B=(0,\frac{e_1}{a}+\alpha+\frac{c_2}{c_1k}), C=(u_*,\frac{e_1}{a}+\alpha+\frac{c_2}{c_1k}), D=(\bar{u},g(\bar{u})),\]
where $u_*$ is such that $g(u_*)=\frac{e_1}{a}+\alpha+\frac{c_2}{c_1k}$ and 
\[c_1=\frac{1-e_1}{e_1}, \, c_2=\frac{e_1}{a}(1-\frac{e_1}{ae_2}).\] 
Let $\gamma'$ be the closed curve defined by:
\[\gamma'=[A,B]\cup [B,C] \cup \zeta \cup [D,A]\]
where,
\[\zeta= \{(u,g(u)); \bar{u}\leq u \leq u_*\}.\]
\begin{theorem}
\label{th:maintheorem}
All the trajectories not in $u=0$ and $v=0$, and different from the fixed point $P_4$ evolve asymptotically towards a unique limit-cycle $\gamma$ which is $O(\epsilon)$ close of $\gamma'$.
\end{theorem}
\begin{proof}
The existence of the cycle results from Poincare-Bendixon theorem. For uniqueness, we refer to \cite{Da04}. The approximation by $\gamma'$ results from slow-fast analysis and the blow-up technique which will be carried out in sections 2 and 3. 
\end{proof}
\begin{remark}
According to  \cite{BC81,KS01,SW01}, the canard phenomenon occurs when a trajectory crosses a folded point from the attractive manifold and  follows the repulsive manifold during a certain amount of time before going away. We will see that according to this definition, the canard phenomena occurs here. This explains why we have introduced $\alpha$ and $k$.
\end{remark}

\section{Slow-Fast Analysis}
In this section, we proceed to a classical slow-fast analysis, see for example \cite{He10,Jo95,Ka99,KS01}.  We study the layer system and the reduced system.  The layer system is obtained by setting $\epsilon=0$ in system \eqref{DLsystem}.
It reads as,
\begin{equation}\label{LayerDLsystem} \left\{ \begin{array}{rcl}
u_t&=&u\left(1-u\right)-\frac{auv}{u+e_{1}}=F(u,v),\\
v_t&=&0 \hspace{2cm}
\end{array} \right.
\end{equation}
The stationary points of this system are given by:
\begin{equation}\label{CriticalManifold} 
M_0=\{u=0 \mbox{ or } v=\frac{1}{a}(1-u)(u+e_1)=g(u)\}
\end{equation}
The set $M_0$ is called the critical manifold. Outside from a neighborhood of this manifold, for $\epsilon$ small,  regular perturbation theory ensures that trajectories of system \eqref{DLsystem} ar $O(\epsilon)$-close to those  of system \eqref{LayerDLsystem}. The trajectories of system \eqref{LayerDLsystem} are tangent to the $u$-axis, which justifies the name of  ``layer system''. These trajectories are the fast trajectories. 
Furthermore, the Fenichel theory, see \cite{Fe79} or references cited above,  provides the existence of a locally invariant manifold $O(\epsilon)$-close to the critical manifold $M_0$ for compact subsets of $M_0$ where $F'_u(u,v)\neq 0$. Thus,  we have to evaluate $F'_u(u,v)$  on the critical manifold. The parts of $M_0$ where  $F'_u(u,v)<0$ is called the attractive part of the critical manifold. Analogously, the part of $M_0$ where  $F'_u(u,v)>0$ is called the repulsive part of the critical manifold. Now, we compute these subset of $M_0$.
We start our computations with the case $u=0$. We have,
\begin{equation}
F'_u(0,v)=1-\frac{av}{e_1}.
\end{equation}
Therefore,
\begin{equation}
F'_u(0,v)>0 \Leftrightarrow v<\frac{e_1}{a}.
\end{equation}
Now, we deal with the case $v=\frac{1}{a}(1-u)(u+e_1)$. We have
\begin{equation}
F'_u(u,v)=1-2u-av\frac{e_1}{(u+e_1)^2}.
\end{equation}
For $v=\frac{1}{a}(1-u)(u+e_1)$, we obtain,
\begin{equation}
F'_u(u,v)=\frac{u}{u+e_1}(-2u+(1-e_1)).
\end{equation}
Therefore,
\begin{equation}
F'_u(u,g(u))>0 \Leftrightarrow u<\frac{1-e_1}{2}=\bar{u}.
\end{equation}
Finally, the attractive critical manifold $M_{0,a}$ is given by $u=0$ and $v>\frac{e_1}{a}$, or $v=g(u)$ and $\frac{1-e_1}{2}<u\leq 1$:
\[M_{0,a}=\{(0,v);v>\frac{e_1}{a}\}\cup \{(u,g(u));\frac{e_1}{a}<u\leq 1\}.\]
Analogously, the repulsive critical manifold $M_{0,r}$ is given by:
\[M_{0,r}=\{(0,v);0\leq v<\frac{e_1}{a}\}\cup \{(u,g(u));0\leq u <\bar{u}\}.\]
 The non-hyperbolic points of the critical manifold, or fold points, where $F'(u,v)=0$ are   $B=(0,\frac{e_1}{a})$ and $D=(\bar{u},g(\bar{u})).$
Now, we look at the reduced system. The reduced system gives the slow-trajectories ie., the trajectories within the critical manifold which persists for $\epsilon$ small within the locally invariant manifold. It is obtained by setting $\epsilon=0$ after  the change of time $\tau=\epsilon t$ in \eqref{DLsystem}.  It reads as (to avoid complications, we keep the notation with $t$, but it should be with $\tau$),
\begin{equation}\label{ReducedDLsystem}
 \left\{ \begin{array}{rcl}
0&=&u\left(1-u\right)-\frac{auv}{u+e_{1}},\\
v_t&=&
v\left(1-\frac{v}{u+e_{2}}\right).\hspace{2cm}
\end{array} \right.
\end{equation}
For $u=0$, we obtain,
\begin{equation}\label{ReducedDLsystemII}
v_t=v(1-\frac{v}{e_{2}}).
\end{equation}
This implies that 
\begin{equation*}
v_t>0 \Leftrightarrow v<e_2.
\end{equation*}
Note  that $(0,e_2)$ is the fixed point $P_2$ of the original system.
For, $v=g(u)$. We have
\begin{equation*}
\begin{array}{rcl}
v_t&>&0\\
\Leftrightarrow  v(1-\frac{v}{u+e_2})&>&0\\
 \Leftrightarrow v<u+e_2
\end{array}
\end{equation*}
which reads also
\[v_t=g'(u)u_t=g(u)(1-\frac{g(u)}{u+e_2}).\]
Therefore,
\[u_t=\frac{g(u)}{g'(u)}(1-\frac{g(u)}{u+e_2}).\]
The points where $g'=0$ correspond to a jump-point if $g(u) \neq u+e_2$, since in this case,  we have at this point, $u_t=-\infty$.
The analysis of layer and reduced system gives the qualitative behavior of the system outside the neighborhood of the fold-points. Trajectories reach the slow attractive manifold, and follow it according to the dynamics, or are repelled  by the repulsive slow manifold. Furthermore, the behavior near the  jump-point $(\bar{u},g(\bar{u})),$ has been rigorously described in \cite{KS01}. Trajectories reaching a neighborhood of the fold point from the right exit the neighborhood at left along fast fibers, and there is a contraction of rate $e^{-\frac{c}{\epsilon}}$ for some constant $c$ between arriving and exiting  trajectories.  The figure \ref{fig:limcyclpdf} illustrates this behavior. 
Therefore, it remains only to analyze the behavior of trajectories near the fold point $P=(0,\frac{e_1}{a})$. This is what we wish to do in the following section by using the blow-up technique. Note that this has not been done in \cite{KS01} since it is assumed there that critical manifold can be written $v=\varphi(u)$ with $\varphi'(0)=0$ and $\varphi''(0)\neq0$, which is not the case here since  $M_0$ writes $u=0$ in a neighborhood of the fold-point $P=(0,\frac{e_1}{a})$.
\begin{remark}
Canards may appear near the fold point $D=(\bar{u},g(\bar{u})),$  when 
\begin{equation}
\label{eq:sing}
g(u)\simeq u+e_2
\end{equation}. As we have already mentioned, canards are solutions that follow the repulsive manifold during a certain amount of time after crossing the fold before being repelled. They have been discovered by french mathematicians with non standard analysis and studied after with geometrical singular perturbation theory, see \cite{BC81,KS01,SW01}.   Our assumptions  prevent the apparition of canards near $D$. Near $P=(0,\frac{e_1}{a})$, we have canards as it is stated in theorem \ref{th:maintheorem}  The condition $e_2\simeq \frac{e_1}{a}$, which is the analog of \eqref{eq:sing} for $P$ would lead to a higher singularity. We don't consider this case here and leave it for a forthcoming work.
\end{remark}

\section{Blow-up technique near the fold-point $P=(0,\frac{e_1}{a})$.} 
The following proposition gives the formulation of \eqref{DLsystem} when written around $(0,\frac{e_1}{a})$:
\begin{proposition}
Near the fold point $(0,\frac{e_1}{a})$, system \eqref{DLsystem} rewrites:
\begin{equation}
\label{eq:arroundfoldpointwithc1}
\begin{array}{rcl}
\dot{x}&=&(c_1x^2-\frac{a}{e_1}xy+O(||(x,y)||^3)\\\
\dot{y}&=&\epsilon(c_2+ \frac{e_1^2}{a^2e_2^2}x+(1-\frac{2e_1}{ae_2})y+O(||(x,y)||^2)\\
\dot{\epsilon}&=&0
\end{array}
\end{equation}
where
\[c_1=\frac{1-e_1}{e_1}, \, c_2=\frac{e_1}{a}(1-\frac{e_1}{ae_2}).\]
\end{proposition}
\begin{proof}
We start with the change of variables 
\[u=x,\, v=\frac{e_1}{a}+y.\]
Plugging into \eqref{DLsystem} gives:
\begin{equation*}
\begin{array}{rcl}
\dot{x}&=&x(1-x)-a\frac{x}{e_1+x}(\frac{e_1}{a}+y)\\
\dot{y}&=&\epsilon(\frac{e_1}{a}+y) (1-\frac{e_1}{a(x+e_2)}-\frac{y}{x+e_2})\\
\dot{\epsilon}&=&0.
\end{array}
\end{equation*}
Then, we use the following Taylor development:
\begin{equation*}
\frac{1}{e_1+x}=\frac{1}{e_1}-\frac{1}{e_1^2}x+\frac{1}{e_1^3}x^2+o(x^2).
\end{equation*}
We find,
\begin{equation}
\label{eq:arroundfoldpoint}
\begin{array}{rcl}
\dot{x}&=&(\frac{1}{e_1}-1)x^2-\frac{a}{e_1}xy+O(x^3)+O(x^2y),\\\
\dot{y}&=&\epsilon(\frac{e_1}{a}(1-\frac{e_1}{ae_2})+ \frac{e_1^2}{a^2e_2^2}x+(1-\frac{2e_1}{ae_2})y+O(||(x,y)||^2)\\
\dot{\epsilon}&=&0,
\end{array}
\end{equation}
which gives the result.\\
Note that $c_1>0$ whereas $c_2<0$.
\end{proof}
We will now apply the blow-up technique. The blow-up technique is a change of variables which allows to desingularize the fold-point and visualize the trajectories in different charts.   We use the following  change of variables:
\[x=\bar{r}\bar{x}, y=\bar{r}^2\bar{y}, \epsilon=\bar{r}^3\bar{\epsilon} \] 
We obtain (we drop the bar):
\begin{equation*}
\label{eq:afterBlowUp}
\begin{array}{rcl}
\dot{r}x+r\dot{x}&=&c_1r^2x^2-\frac{a}{e_1}r^3xy+O(r^4x^2y)+O(r^3x^3)\\
2ry\dot{r}+r^2\dot{y}&=&r^3\epsilon(c_2+ \frac{e_1^2}{a^2e_2^2}rx+(1-\frac{2e_1}{ae_2})r^2y+O(||(rx,r^2y)||^2))\\
3r^2\epsilon\dot{r}+r^3\dot{\epsilon}&=&0
\end{array}
\end{equation*}
The chart $K_1$ is obtained by setting $\bar{y}=1$.
The chart $K_2$ is obtained by setting $\bar{\epsilon}=1$.
The chart $K_3$ is obtained by setting $\bar{x}=1$.\\
In order to prove theorem we need only to consider the chart $K_2$ which will be fundamental in our analysis.
When working ni chart $K_2$, we use the suscript $2$.\\
\textbf{Dynamics in chart $K_2$. }\\
\begin{proposition}
The dynamics in chart $K_2$ are given by the system:
\begin{equation}
\label{eq:K2}
\begin{array}{rcl}
\dot{x}_2&=&c_1x^2_2+O(r_2))\\
\dot{y}_2&=&c_2+O(r_2)\\
\dot{r}_2&=&0
\end{array}
\end{equation}
\end{proposition}
\begin{proof}
Setting $\bar{\epsilon}=1$ in \eqref{eq:afterBlowUp} gives:
\begin{equation*}
\begin{array}{rcl}
\dot{x}_2&=&r_2(c_1x_2^2+O(r_2))\\
\dot{y}_2&=&r_2(c_2+O(r_2))\\
\dot{r}_2&=&0.
\end{array}
\end{equation*}
Then, we desinguralize the system  by a change of time $\tau=r_2 t$, which gives the result. 
\end{proof}
For $r_2=0$, we obtain: 
\begin{equation}
\label{eq:K2r=0}
\begin{array}{rcl}
\dot{x}_2&=&c_1x_2^2\\
\dot{y}_2&=&\frac{e_1}{a}(1-\frac{e_1}{ae_2})\\
\dot{r}_2&=&0.
\end{array}
\end{equation}
Equation \eqref{eq:K2r=0} is very important in our analysis since it shows how the trajectories cross the fold point.

\begin{proposition}
 The solution of system \eqref{eq:K2r=0} is:
\begin{equation}
\label{eq:sol:K2r=0}
\begin{array}{rcl}
x_2(t)&=&\frac{1}{x_2^{-1}(0)-c_1t}\\
y_2(t)&=&y_2(0)+c_2t\\
\end{array}
\end{equation}
i.e.
\begin{equation*}
\begin{array}{rcl}
x_2(t)&=&\frac{1}{x_2^{-1}(0)-c_1\frac{y_2(t)-y_2(0)}{c_2}}
\end{array}
\end{equation*}
or
\begin{equation*}
\begin{array}{rcl}
y_2(t)&=&y_2(0)+\frac{c_2}{c_1}(\frac{1}{x_2(0)}-\frac{1}{x_2(t)})
\end{array}
\end{equation*}
It follows that orbits have the following properties:
\begin{enumerate}
\item Every orbit has a horizontal asymptote $y = y_r$, where $y_r$ depends on the orbit
such that $x\rightarrow +\infty$ as $y$ approaches $y_r$ from above.
\item  Every orbit has a  vertical asymptote $x= 0^{+}$.
\item The point $(x_2(0),\alpha,0)$ is mapped to the point $(\delta, \alpha+\frac{c_2}{c_1}(\frac{1}{x_2(0)}-\frac{1}{\delta}))$.
\end{enumerate}

\end{proposition}
\begin{proof}
It follows easily from the explicit solution.
\end{proof}

\begin{proposition}
Solutions of \eqref{eq:K2} are $O(r)$-close of those of \eqref{eq:K2r=0}.
\end{proposition}
\begin{proof}
This follows from regular perturbation theory.
\end{proof}
\begin{remark}
Let us make a remark on the first statement of proposition 3.
For $t^*=\frac{1}{c_1x_2(0)}$, $x_2$ blows-up. Since $x_2=\frac{x}{r_2}$, and $r_2=\epsilon^{\frac{1}{3}}, x_2=+\infty $ correspond, when $\epsilon=0$ to a point $x>0$ where we can consider that trajectory has left the neighborhood of the fold and where the previous slow-fast analysis applies. This gives for $y_2$:
\begin{equation}
\label{eq:valueconnect}
y_2(t^*)=y_2(0)+\frac{c_2}{c_1x_2(0)}.
\end{equation}
This means, that fixing $x_2(0)$ and $y_2(0)$, the value where the trajectory leaves the slow manifold and connects the fast fiber is determined by \eqref{eq:valueconnect}. Therefore, if we choose $(x_2(0),y_2(0))$ on the limit-cycle, this determines the fast fiber followed by the limit-cycle. We will now detail this argument which gives the proof of theorem \ref{th:maintheorem}.  
\end{remark}
\begin{proof}[proof of theorem \ref{th:maintheorem}]
Fix a value $x$ far from $0$, let's say $x=\frac{1}{2}$. We want to determine $t^*$ such that $x(t^*)=\frac{1}{2}$, which corresponds to $x_2(t^*)=\frac{1}{2\epsilon^{\frac{1}{3}}}$. Taking $x(0)=k\epsilon+o(\epsilon)$, and  according to equation \eqref{eq:sol:K2r=0}, this gives:
\[t^*=\frac{\epsilon^\frac{1}{3}}{c_1}(\frac{1}{k\epsilon+o(\epsilon)}-2)\]
and for equation \eqref{eq:K2},
\[y_2(t^*)=y_2(0)+\frac{c_2\epsilon^\frac{1}{3}}{c_1}(\frac{1}{k\epsilon+o(\epsilon)}-2) +O(\epsilon),\]
which in original coordinates gives:
\[y(t^*)=y(0)+\frac{c_2}{kc_1}+O(\epsilon).\]
This proves the theorem.
\end{proof}

\begin{figure}[ht]
 \centering
  \includegraphics[scale=0.6]{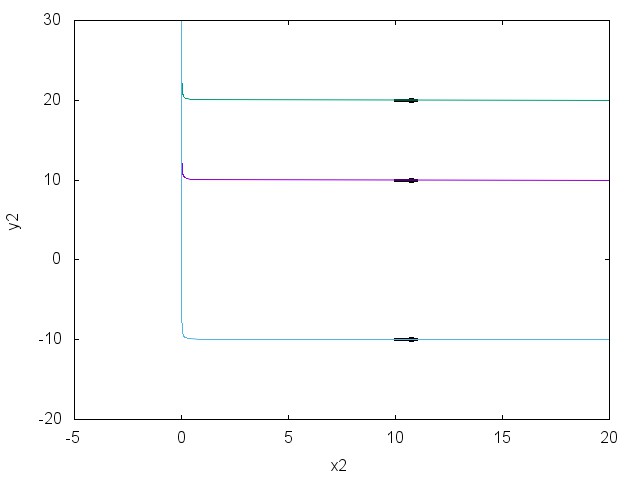}
  \label{fig:ChartK2}
\caption{Solutions of system \eqref{eq:K2r=0} (chart $K_2$) for  $a=1$
, $e_{1}=0.08,$ and $e_{2}=0.01$. Since we focus on $x>0$, we have only represented solutions on the half right plane. These solutions are defined on interval $]-\infty;\frac{1}{c_1x_2(0)}[$}
\end{figure}

\begin{figure}[ht]
 \centering
  \includegraphics[scale=0.6]{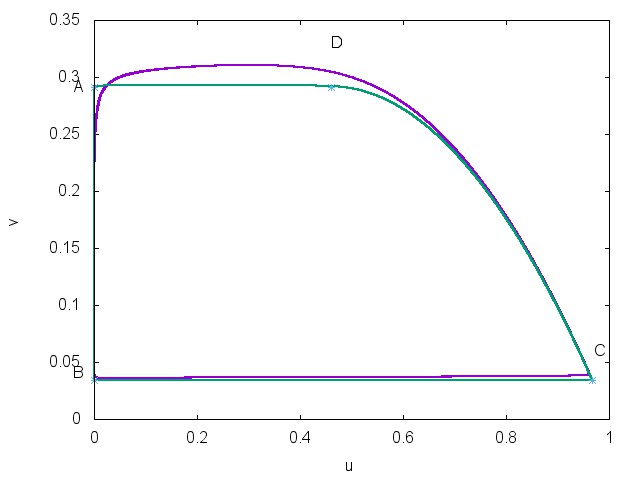}
  \label{figlimcyclvareps}
\caption{Solutions of system \eqref{DLsystem} with $a=1$
, $e_{1}=0.08,$ and $e_{2}=0.01$. Limit-cycles for $\epsilon=10^{-2}$ in purple and  $\epsilon=10^{-4}$ in green. As $\epsilon$ approaches $0$ the limit-cycle approaches $\gamma'$. Choosing $\alpha=0$, we obtain an approximate value of $k\simeq -\frac{c_2}{Dc_1}$ where $D$ is the distance between the coordinate of the fast-fiber followed by the limit cycle and $\frac{e_1}{a}$. Here, we find $D\simeq 0.045$ and $k\simeq 1.086 $, $(c_1=11.5,c_2=-0.56)$. This means that the limit-cycle crosses the axis $v=\frac{e_1}{a}$ at a value $1.086\epsilon+o(\epsilon)$. }
\end{figure}
\begin{remark}
Note that the folded node $P$ is at the intersection of the the two branches of the  manifold $M_0$, $v=g(u)$ and $u=0$. Note also that these two branches actually exchange their stability at $P$. This case has been treated in a general form in \cite{KS01-2} under the appropriate name of transcritical bifurcation. However, here we are precisely in the special case $\lambda=1$ excluded from theorem 2.1 of \cite{KS01-2}. The authors have announced the existence of the canard in this case without giving the detailed proof of it. Here, we have proved the canard phenomenon using the blow up technique in  the case of the limit-cycle of this classical model of predator-prey. 
\end{remark}

\section{Conclusion}
In this article, we have characterised the limit-cycle of the system  \eqref{DLsystem}. The system was originally introduced in \cite{az03}
as a modification of the Leslie-Gower model. We have proved that the limit-cycle of the model exhibits the canard phenomenon when crossing a special folded node as well as computed the value at which it reaches the fast fiber. In a forthcoming work, we hope to investigate the diffusive model obtained by adding a laplacian term in the first equation.

\begin{acknowledgements}
We would like to thank Region Haute-Normandie France and the 
ERDF (European Regional Development Fund)  project 
XTERM (previously RISK). We would like thank N. Popovic for discussions on the transcritical bifurcation phenomenon.  
\end{acknowledgements}


\begin{thebibliography}{}
 \bibitem{az03} M.A. Aziz-Alaoui and M. Daher Okiye, Boundedness
and global stability for a predator-prey model with modified
Leslie-Gower and Holling-type $II$ schemes, Appl. Math. Lett. 16
(2003) 1069-1075.
 \bibitem{Da04} M. Daher Okiye, \'Etude et analyse asymptotique de certains syst\`emes dynamiques non-lin\'eaires : application \`a des probl\`emes proie-pr\'edateurs. PhD thesis, Le Havre, 2004.
 \bibitem{BC81} E.. Benoit, J.-L. Callot, F. Diener and M. Diener, Chasse au canards, Collect. Math., 31-32 (1981) 37-119.
 \bibitem{Fe79} N. Fenichel, Geometric singular perturbation theory for ordinary differential equations. J. Differ. Equ. 31   (1979) 53-98.
 \bibitem{Ha91} I.L. Hanski, L. Hassen,  and H. Huttonen, Specialist Predation, generalist predation and the rodent microtine cycle, J. Animal Ecology 60 (1991) 353-367. 
(1992) 237-388.
 \bibitem{He10} G. Hek, Geometric singular perturbation theory in biological pratice, J. Math. Biol. 60
(2010) 347-386.
 \bibitem{Jo95} C.K.R.T. Jones 
Geometric singular perturbation theory. In: Johnson R (ed) Dynamical systems, Montecatibi Terme, Lecture Notes in Mathematics, Springer, Berlin. 1609 (1995)  44-118.
 \bibitem{Ka99} T.J. Kaper   An introduction to geometric methods and dynamical systems theory for singular perturbation problems. In: Cronin J, O’Malley RE Jr (eds) Analyzing multiscale phenomena using singular perturbation methods. Proc Symposia Appl Math, AMS, Providence,  56 (1999) 85-132.  
 \bibitem{KS01} M. Krupa and P. Szmolyan, Extending geometric singular perturbation theory to non-hyperbolic points-fold and canard points in two dimensions, SIAM. J. Math. Anal. 33 (2001) 286-314.
 \bibitem{KS01-2} M. Krupa and P. Szmolyan, Extending slow manifolds near transcritical and
pitchfork singularities
, Nonlinearity 14 (2001) 1473-1491.
 \bibitem{Le48} P.H. Leslie,  Some further notes on the use of matrices in population mathematics, Biometrica  35 (1948) 213-245.
 \bibitem{Le60} P.H. Leslie and J.C. Gower,  The properties of a stochastic model for the predator-prey type of interaction
between two species, Biometrica 47 (1960) 219-234.
\bibitem{Ma73}May R.M. Stability and complexity in model ecosystems. Princeton, NJ: Princeton University Press (1973).
\bibitem{RM63}Rosenzweig M.L. and MacArthur R.H. Graphical representation and stability conditions of predator–prey interaction. Amer.
Naturalist. 47 (1963) 209-223
 \bibitem{RM92} S. Rinaldi and S. Muratori, Slow-fast limit-cycles in predator-prey models, Ecological Modelling. 61
(1992) 237-388.
 \bibitem{SW01} P. Szmolyan and M. Wechselberger, Canards in R3, J. Differential Equations, 177 (2001) 419-453.
 \bibitem{Up97} R.K. Upadhyay and V. Rai, Why chaos is rarely observed in natural populationss, Chaos Solitons and Fractals. 8(12) (1997) 1933-1939. 
 \bibitem{Wi94} S. Wiggins Normally hyperbolic invariant manifolds in dynamical systems. Springer, New York (1994).




\end{thebibliography}


\end{document}